\documentclass[final]{ws-ijnt}
\usepackage{mathrsfs}

\newcommand{\Parans}[1]{\left(#1\right)}
\newcommand{\CBrackets}[1]{\left\{#1\right\}}
\newcommand{\SBrackets}[1]{\left[#1\right]}
\newcommand{\aqprod}[3]{\Parans{#1;#2}_{#3}}
\newcommand{\Jac}[2]{\left(\frac{#1}{#2}\right)}
\newcommand{\spt}[1]{\mbox{\normalfont spt}\Parans{#1}}
\newcommand{\InTheoremCite}[2]{\mbox{\normalfont (#1 \cite{#2})}}
\newcommand{\SLift}[3]{\mathscr{S}_{#1,#2}\Parans{#3}}
\begin{document}

\markboth{Frank G. Garvan, Chris Jennings-Shaffer}
{Andrews' spt-function modulo 16 and 32}

%
%

\title{HECKE-TYPE CONGRUENCES FOR ANDREWS' SPT-FUNCTION MODULO 16 AND 32}

\author{FRANK G. GARVAN}

\address{Department of Mathematics, University of Florida\\
Gainesville, Florida 32611, USA\\
\email{fgarvan@ufl.edu}}

\author{CHRIS JENNINGS-SHAFFER}

\address{Department of Mathematics, University of Florida\\
Gainesville, Florida 32611, USA\\ 
\email{cjenningsshaffer@ufl.edu}}

\maketitle

\begin{abstract}
Inspired by recent congruences by Andersen with varying powers of $2$ in the modulus for partition related functions, we extend the modulo $32760$ congruences of the first author for the function $\spt{n}$. We show that a normalized form of the generating function of $\spt{n}$ is an eigenform modulo $32$ for the Hecke operators $T(\ell^2)$ for primes $\ell\ge 5$ with $\ell\equiv 1,11,17,19\pmod{24}$, and an eigenform modulo $16$ for $\ell\equiv 13,23\pmod{24}$. 
\end{abstract}



\section{Introduction and Statement of Results}

In this paper we strengthen a congruence on the smallest parts function. We recall that Andrews in \cite{Andrews} defined the function 
$\spt{n}$ as the number of smallest parts in the partitions of $n$. This function is known to satisfy many interesting and striking congruences, we list a few. Andrews proved
\begin{align} 
	\spt{5n + 4} &\equiv 0 \pmod{5},\\
	\spt{7n + 5} &\equiv 0 \pmod{7},\\
	\spt{13n + 6} &\equiv 0 \pmod{13}.
\end{align}
The first author proved in \cite{Garvan3} for $a,b,c\ge 3$ and $\delta_a,\lambda_b,\gamma_c$ the least nonnegative residues of the reciprocals of $24$ modulo $5^a$, $7^b$, and $13^c$ respectively, that
{\allowdisplaybreaks\begin{align}
	\spt{5^an + \delta_a} + 5\spt{5^{a-2}n + \delta_{a-2}} &\equiv 0 \pmod{5^{2a-3}},\\
	\spt{7^bn + \lambda_b} + 7\spt{7^{b-2}n + \lambda_{b-2}} &\equiv 0 \pmod{7^{\lfloor\frac{1}{2}(3b-2)\rfloor}},\\
	\spt{13^cn + \gamma_c} - 13\spt{13^{c-2}n + \gamma_{c-2}} &\equiv 0 \pmod{13^{c-1}}.
\end{align}}
In \cite{Ono3} Ono proved for any prime $\ell\ge 5$  that 
\begin{align}\label{eq1}
	\spt{\frac{\ell^{3}n +1}{24}}  
	& \equiv
	\Jac{3}{\ell}\spt{\frac{\ell n + 1}{24}} 
 \pmod{\ell}.
\end{align}
The congruence in (\ref{eq1}) was generalized in \cite{ABL} by Ahlgren, Bringmann, and Lovejoy to
\begin{align}\label{ABLCongruence}
	\spt{\frac{\ell^{2m+1}n +1}{24}} 
	& \equiv 
	\Jac{3}{\ell}\spt{\frac{\ell^{2m-1}n+1}{24}} 
\pmod{\ell^m}
\end{align}
for $m\ge 1$ and prime $\ell\ge 5$. This was accomplished by relating the Hecke operator $T(\ell^{2m})$ to the Hecke operators for lower powers of $\ell$. Recently Ahlgren and Kim in \cite{AK} have explained that the congruence in (\ref{ABLCongruence}) and others actually come from equalities, rather than congruences, between the images of certain mock modular forms under Hecke operators. This is done by a systematic study of how Hecke operators transform certain mock modular grids. 

To state our results, we first need some definitions. As in \cite{Ono3}, \cite{Garvan3}, \cite{Garvan1} we define
\begin{align}
	\mathbf{a}(n) := 12\spt{n} + (24n - 1)p(n),
\end{align}
for $n \ge 0$, and
\begin{align}
	\alpha(z) := \sum_{n=0}^\infty \mathbf{a}(n)q^{n-\frac{1}{24}},
\end{align}
where $q = \exp(2\pi iz)$ and $Im(z) > 0$. We let $\chi_{12}$ be the primitive Dirichlet character modulo $12$ given by
\begin{align}
	\chi_{12}(n) &= 	
	\left\{
   	\begin{array}{lll}
      	1 & \hspace{15pt}\mbox{ if } n\equiv \pm 1 \pmod{12}\\
       	-1 & \hspace{15pt}\mbox{ if } n\equiv \pm 5 \pmod{12}\\
			0 & \hspace{15pt}\mbox{ otherwise}.
     	\end{array}
	\right.
\end{align}
The Dedekind eta function is given by
\begin{align}
	\eta(z):=q^{\frac{1}{24}}\prod_{n=1}^\infty(1-q^n). 
\end{align}
As in \cite{Bringmann} we set 
\begin{align}
\mathcal{M}(z):= \alpha(24z) -\frac{3i}{\pi\sqrt{2}}\int_{-\overline{z}}^{i\infty}
	\frac{\eta(24\tau d\tau)}{(-i(\tau+z))^{\frac{3}{2}}}.
\end{align}
In \cite{Bringmann} Bringmann showed $\mathcal{M}(z)$ is a weight $\frac{3}{2}$ weak Maass form  on $\Gamma_0(576)$ with Nebentypus $\chi_{12}$ and 
that $\alpha(24z)$ is the holomorphic part of $\mathcal{M}(z)$. 

For a Dirichlet character $\chi$ and prime $\ell$, we have the weight $\frac{3}{2}$ Hecke operator $T_\chi(\ell^2)$. We know $T_\chi(\ell^2)$ operates on $q$-series expansions by
\begin{align}
	\Parans{\sum b(n)q^n}\mid T_\chi(\ell^2)
	&= 
	\sum\Parans{
		b(\ell^2n)
		+\chi(\ell)\Jac{-n}{\ell}b(n)
		+\chi(\ell^2)\ell b(n/\ell^2)   
 		}q^n.
\end{align}

In $\cite{Ono3}$ Ono showed that for $\ell \ge 5$ prime, the operator
\begin{align}
T_{\chi_{12}}(\ell^2) - \chi_{12}(\ell)\ell(1+\ell)
\end{align}
annihilates the nonholomorphic part of $\mathcal{M}(z)$. Furthermore, the function $\mathcal{M}_\ell(z/24)$ is a weakly holomorphic
modular form of weight $\frac{3}{2}$ for the full modular group, where
\begin{align}
	\mathcal{M}_\ell(z) 
	&:= 
	M(z)\mid T_{\chi_{12}}(\ell^2)  - \chi_{12}(\ell)(1+\ell)\mathcal{M}(z)
	\nonumber\\ 
	&= \alpha(24z)\mid T_{\chi_{12}}(\ell^2) - \chi_{12}(\ell)(1 + \ell)\alpha(24z).
\end{align}
Additionally he proved the following two theorems. 
\begin{theorem}\label{TheoremModularM}$\InTheoremCite{Ono}{Ono3}$
If $\ell \ge 5$ is prime then the function $\mathcal{M}_\ell(z/24)\eta(z)^{\ell^2}$ is an entire modular form of weight $\frac{1}{2}(\ell^2+3)$
for the full modular group $\Gamma(1)$.
\end{theorem}

\begin{theorem}\label{OnoCongruence}$\InTheoremCite{Ono}{Ono3}$
If $\ell \ge 5$ is prime then
\begin{align}
&\spt{\ell^2n - s_\ell} + \chi_{12}(\ell)\Jac{1 - 24n}{\ell}\spt{n} + \ell\spt{\frac{n+s_\ell}{\ell^2}}
\nonumber\\&\equiv \chi_{12}(\ell)(1+\ell)\spt{n} \pmod{3},
\end{align}
where 
\begin{align}\label{DefOfSl}
s_\ell &= \frac{\ell^2-1}{24}.
\end{align}
\end{theorem}

Theorem \ref{OnoCongruence} was conjectured by the first author and the first author extended this to the following. 
\begin{theorem}$\InTheoremCite{Garvan}{Garvan1}$\begin{enumerate}
\item If $\ell \ge 5$ is prime then
\begin{align}
&\spt{\ell^2n - s_\ell} + \chi_{12}(\ell)\Jac{1 - 24n}{\ell}\spt{n} + \ell\spt{\frac{n+s_\ell}{\ell^2}}
\nonumber\\\label{toBeImproved}
&\equiv \chi_{12}(\ell)(1+\ell)\spt{n} \pmod{72}.
\end{align}
\item If $\ell\ge 5$ is prime, $t=5,7,$ or $13$ and $\ell\not=t$ then
\begin{align}
&\spt{\ell^2n - s_\ell} + \chi_{12}(\ell)\Jac{1 - 24n}{\ell}\spt{n} + \ell\spt{\frac{n+s_\ell}{\ell^2}}
\nonumber\\
&\equiv \chi_{12}(\ell)(1+\ell)\spt{n} \pmod{t}.
\end{align}
\end{enumerate}
\end{theorem}

In \cite{Andersen} Andersen considered congruences for various partition related functions and found the power of $2$ in the modulus can vary. For example, with $\overline{\mbox{spt}1}(n)$ denoting the number of smallest parts in the overpartitions of $n$ having smallest part odd, he proved the following congruence.

\begin{theorem} Let $\ell$ be an odd prime, and define
\begin{align}
	\alpha &:= 
		\left\{\begin{array}{lll}
			6 & \mbox{ if } & \ell \equiv 3 \pmod{8}  
			\\
			7 & \mbox{ if } & \ell \equiv 5,7 \pmod{8}
			\\
			8 & \mbox{ if } & \ell \equiv 1 \pmod{8}.
		\end{array}\right.
\end{align}
Then for $t\in\{2^\alpha,3,5\}$, $\ell\not=t$, and $n\ge 1$, we have
\begin{align}
	\overline{\mbox{\normalfont spt1}}(\ell^2 n)
	+\Jac{-n}{\ell}\overline{\mbox{\normalfont spt1}}(n)
	+\ell\overline{\mbox{\normalfont spt1}}(n/\ell^2)
	&\equiv	
	(1+\ell)\overline{\mbox{\normalfont spt1}}(n) \pmod{t}.	
\end{align}
\end{theorem}

Similar to this, we find the power of $2$ in the modulus of (\ref{toBeImproved}) can be increased for certain values of $\ell$. Specifically, we will prove the following theorem.
\begin{theorem}\label{MainTheorem} Let $\ell\ge 5$ be prime, and define
\begin{align}
	\beta &:= 
		\left\{\begin{array}{lll}
			3 & \mbox{ if } & \ell \equiv 7,9 \pmod{24}  
			\\
			4 & \mbox{ if } & \ell \equiv 13,23 \pmod{24}  
			\\
			5 & \mbox{ if } & \ell \equiv 1,11,17,19 \pmod{24}.
		\end{array}\right.
\end{align}
Then for $n\ge 1$ we have
\begin{align}
&\spt{\ell^2n - s_\ell} + \chi_{12}(\ell)\Jac{1 - 24n}{\ell}\spt{n} + \ell\spt{\frac{n+s_\ell}{\ell^2}}
\nonumber\\
&\equiv \chi_{12}(\ell)(1+\ell)\spt{n} \pmod{2^\beta}.
\end{align}
\end{theorem}
The case when $\ell\equiv 7,9\pmod{24}$ is already given by (\ref{toBeImproved}). If we let the generating function for the spt function be given by
\begin{align}
	\mbox{SPT}(z) := \sum_{n=1}^\infty \spt{n}q^{n-\frac{1}{24}},
\end{align}
then Theorem \ref{MainTheorem} is saying that $\mbox{SPT}(24z)$ an eigenform modulo $2^\beta$ for the Hecke operators $T_{\chi_{12}}(\ell^2)$ with eigenvalue $\chi_{12}(\ell)(1+\ell)$.

\section{Preliminaries}\label{prelim}
As in \cite{Garvan1}, we make use of the following definitions and theorems. For even integers $k\ge 2$ we denote by $E_{k}(z)$ the normalized Eisenstein series for the full modular group,
\begin{align}
	E_{k}(z) &:= 1 - \frac{2k}{B_{k}}\sum_{n=1}^\infty \sigma_{k-1}(n)q^n ,
\end{align}
where $q=\exp(2\pi iz)$, $B_{k}$ is the $k^{th}$ Bernoulli number, and $\sigma_{k-1}$ is the sum of the divisors function given by
\begin{align}
	\sigma_{k-1}(n) &:= \sum_{d\mid n} d^{k-1}.
\end{align}
For $k>2$, $E_k$ is a modular form of weight $k$ for the full modular group. In particular the first few Eisenstein series are given by
\begin{align}
	\label{E2Def}
	&E_2(z) = 1 - 24\sum_{n=1}^\infty \sigma_1(n)q^n, 
	\\\label{E4Def}
	&E_4(z) = 1 + 240\sum_{n=1}^\infty \sigma_3(n)q^n,
	\\\label{E6Def}
	&E_6(z) = 1 - 504\sum_{n=1}^\infty \sigma_5(n)q^n.
\end{align}

We also use Ramanujan's function $\Delta(z)$, 
\begin{align}
	\Delta(z) := \eta(z)^{24} = q\prod_{n=1}^\infty(1-q^n)^{24},
\end{align}
and Klein's modular invariant
\begin{align}
	j(z) := \frac{E_4(z)^3}{\Delta(z)} = q^{-1} + 744 + 196884q + \dots.
\end{align}

We make use the standard infinite product notation
\begin{align}
	\aqprod{q}{q}{\infty} := \prod_{n=1}^\infty(1-q^n) = q^{-\frac{1}{24}}\eta(z).
\end{align}

By $M_{k}(\Gamma,\chi)$ we mean the finite dimensional vector space of holomorphic modular forms of weight $k$ on a congruence subgroup $\Gamma$ of $SL(2,\mathbb{Z})$ with character $\chi$. By $S_{k}(\Gamma,\chi)$ we mean the subspace of cusp forms of $M_{k}(\Gamma,\chi)$. When $\chi$ is the trivial character, we instead write $M_{k}(\Gamma)$ and $S_{k}(\Gamma)$.

Next we define the series with which we will work. We define 
\begin{align}
	d(n) := (24n - 1)p(n),
\end{align}
so that
\begin{align}
	\sum_{n=0}^\infty d(n)q^{24n-1} &= q\frac{d}{dq}\frac{1}{\eta(24z)}
	= -\frac{E_2(24z)}{\eta(24z)},
	\\
	\mathbf{a}(n) &= 12\spt{n} + d(n).
\end{align}
For $\ell \ge 5$ prime we define
{\allowdisplaybreaks
\begin{align}
	\label{DefOfZ}
	&Z_\ell(z) :=
	\sum_{n=-s_\ell}^\infty{\textstyle 
		\Parans{	\ell^3p(\ell^2n - s_\ell) 
					+ \ell\chi_{12}(\ell)\Jac{1-24n}{\ell}p(n) 
					+ p\Parans{\frac{n+s_\ell}{\ell^2}}
		}
		q^{n - \frac{1}{24}}},
	\\
	&\Xi_\ell(z) := 
	\sum_{n=-s_\ell}^\infty{\textstyle
		\Parans{	d(\ell^2n - s_\ell) 
				+ \chi_{12}(\ell)\Parans{\Jac{1-24n}{\ell}-1-\ell}d(n) 
				+ \ell d\Parans{\frac{n+s_\ell}{\ell^2}}
		}
		q^{n - \frac{1}{24}}},
	\\
	&\mathcal{A}_\ell(z) :=
	\sum_{n=-s_\ell}^\infty{\textstyle
		\Parans{	\mathbf{a}(\ell^2n - s_\ell) 
			+ \chi_{12}(\ell)\Parans{\Jac{1-24n}{\ell}-1-\ell}\mathbf{a}(n) 
			+ \ell\mathbf{a}\Parans{\frac{n + s_\ell}{\ell^2}}
		}
		q^{n- \frac{1}{24}}},
	\\
	&\mathcal{S}_\ell(z) := 
	\nonumber\\\label{defOfS}
	&\sum_{n=1}^\infty{\textstyle
		\Parans{	\spt{\ell^2n - s_\ell} 
			+ \chi_{12}(\ell)\Parans{\Jac{1-24n}{\ell}-1-\ell}\spt{n} 
			+ \ell\spt{\frac{n + s_\ell}{\ell^2}}
		}
		q^{n- \frac{1}{24}}},
\end{align}
where $s_\ell$ is given in (\ref{DefOfSl}). We note that
\begin{align}\label{AByXiAndS}
\mathcal{A}_\ell(z) = 12\mathcal{S}_\ell(z) + \Xi_\ell(z) = \mathcal{M}_\ell(z/24).
\end{align}

By (\ref{defOfS}) we see to prove Theorem \ref{MainTheorem} we can instead show 
\begin{align}
	\mathcal{S}_\ell(z)&\equiv 0\pmod{16}
\end{align}
for prime $\ell\equiv 13,23\pmod{24}$ and 
\begin{align}
	\mathcal{S}_\ell(z)\equiv 0\pmod{32}
\end{align}
for prime $\ell\equiv 1,11,17,19\pmod{24}$.

We need some properties of the series we've defined. By Theorem \ref{TheoremModularM} and (\ref{AByXiAndS}) we see that the function
\begin{align}\label{preliminariesEq1}
	\ell\mathcal{A}_\ell(z)\eta(z)\Delta(z)^{s_\ell}
	&\in M_{\frac{\ell^2+3}{2}}(\Gamma(1)).
\end{align}
Since $\frac{\ell^2+3}{2} = 2+12s_\ell$ we have the dimension of $M_{\frac{\ell^2+3}{2}}(\Gamma(1))$ is $s_\ell$ and the set
\begin{align}\label{preliminariesEq2}
	\CBrackets{E_4(z)^{3n-1}E_6(z)\Delta(z)^{s_\ell-n} : 1 \le n \le s_\ell}
\end{align}
is a basis. By (\ref{preliminariesEq1}) and (\ref{preliminariesEq2}) we see there exist integers $b_{n,\ell}$, for $1 \le n \le s_\ell$, such that
\begin{align}\label{AByIntegers}
	\ell\mathcal{A}_\ell(z)\eta(z)\Delta(z)^{s_\ell}
	&=
	\sum_{n=1}^{s_\ell}b_{n,\ell}E_4(z)^{3n-1}E_6(z)\Delta(z)^{s_\ell-n}.
\end{align}

Next we consider the following result of Atkin.
\begin{theorem}\label{AtkinTheorem}$\InTheoremCite{Atkin}{Atkin}$The function $Z_\ell(z)\eta(z)$ is a modular function on the full modular
group $\Gamma(1)$.
\end{theorem}
By Theorem \ref{AtkinTheorem} we know $Z_\ell(z)\eta(z)$ is a polynomial in $j(z)$. In \cite{Ono2} Ono determines this polynomial by defining a sequence of polynomials $A_m(x) \in \mathbb{Z}\SBrackets{x}$, each of degree $m$, by
\begin{align}
\sum_{m=0}^\infty A_m(x)q^m &= 
\aqprod{q}{q}{\infty} \frac{E_4(z)^2E_6(z)}{\Delta(z)}
			\frac{1}{j(z)-x}
\\
&= 1 + (x - 745)q + (x^2 - 1489x + 160511)q^2 +\dots.
\end{align} 
He then proves the following theorem.
\begin{theorem}$\InTheoremCite{Ono}{Ono2}$ For $\ell \ge 5$ prime
\begin{align}
Z_\ell(z)\eta(z) = \ell\chi_{12}(\ell) + A_{s_\ell}(j(z)),
\end{align}
where $Z_\ell(z)$ is given in (\ref{DefOfZ}), and $s_\ell$ is given in (\ref{DefOfSl}). 
\end{theorem}

We define a sequence of polynomials $C_\ell(x) \in \mathbb{Z}\SBrackets{x}$ by
\begin{align}
C_\ell(x) &:= \ell\chi_{12}(\ell) + A_{s_\ell}(x)
= \sum_{n=0}^{s_\ell}c_{n,\ell}x^n,
\end{align}
so that $Z_\ell(z)\eta(z) = C_\ell(j(z)).$ 

The proof of Theorem \ref{MainTheorem} will start by handling $S_\ell(z)$ through $A_\ell(z)$ and $\Xi_\ell(z)$. We have the following theorem for $\Xi_\ell(z)$.
\begin{theorem}\InTheoremCite{Garvan}{Garvan1} For $\ell\ge 5$ we have
\begin{align}
	\ell\Xi(z)\eta(z)\Delta(z)^{s_\ell}
	=&
	-\sum_{n=0}^{s_\ell}
	 	c_{n,\ell}E_4(z)^{3n-1}\Delta(z)^{s_\ell-n}(24nE_6(z)+E_4(z)E_2(z))	
	\nonumber\\
	&+\chi_{12}(\ell)(1+\ell)E_2(z)\Delta(z)^{s_\ell}.
\end{align}
\end{theorem}

\section{The Proof of Theorem \ref{MainTheorem}}\label{proof}
To start we claim the following congruences hold 
\begin{align}
	&E_4(z)^8 \equiv 1 \pmod{128},
	\\\label{E2Congruence}
	&E_2(x) \equiv E_4(z)^7 E_6(z) + 80\Delta(z) + 32\Delta(z)^2 
				+ 64\Delta(z)^4 \pmod{128}.
\end{align}
The first congruence is clear. The congruence in (\ref{E2Congruence}) is realized by finding both sides to be congruent to modular forms of weight $258$ and using Sturm's theorem. That the right hand side is congruent to a weight $258$ modular form on $\Gamma(1)$ is immediate. For the left hand side we use the modular form $-\sum_{n=0}^6 2^{n}E_{258}(2^nz)$, which is a weight $258$ modular form on $\Gamma_0(64)$. The power of $q$ specified by Sturm's theorem is $2064$ and the verification is done in Maple.

We can use this to compute $\ell\Xi_\ell(z)\eta(z)\Delta(z)^{s_\ell} \pmod{128}$. We also use the following congruences which follow immediately from (\ref{E2Def}), (\ref{E4Def}), and (\ref{E6Def}), 
\begin{align}
	&80\equiv 80E_2(z)\equiv 80E_4(z) \equiv 80 E_6(z) \pmod{128},
	\\
	&32\equiv 32E_2(z) \equiv 32E_4(z) \equiv 32 E_6(z) \pmod{128},
	\\
	&64\equiv 64E_2(z) \equiv 64E_4(z) \equiv 64 E_6(z) \pmod{128}. 
\end{align}
For brevity we suppress the dependence on $z$. We have
\begin{align}\label{XiMess}
\ell&\Xi_\ell(z)\eta(z)\Delta(z)^{s_\ell}
\nonumber\\
=&
	-\sum_{n=0}^{s_\ell}
		c_{n,\ell}E_4^{3n-1}\Delta^{s_\ell-n}(24nE_6+E_4E_2)
	+\chi_{12}(\ell)\ell(1+\ell)E_2\Delta^{s_\ell}
\nonumber\\
\equiv&
	-\sum_{n=1}^{s_\ell}
		c_{n,\ell}E_4^{3n-1}\Delta^{s_\ell-n}
		(24nE_6 + 4E_6 + 80E_4\Delta + 32E_4\Delta^2 + 64E_4\Delta^4)
	\nonumber\\
	&+(\chi_{12}(\ell)\ell(1+\ell)-c_{0,\ell})E_2\Delta^{s_\ell}
	\pmod{128}
\nonumber\\
\equiv&
	-\sum_{n=1}^{s_\ell}
		(24n+1)c_{n,\ell}E_4^{3n-1}E_6\Delta^{s_\ell-n}
	-\sum_{n=0}^{s_\ell-1}
		80c_{n+1,\ell}E_4^{3n+3}\Delta^{s_\ell-n}
	\nonumber\\
	&-\sum_{n=-1}^{s_\ell-2}
		32c_{n+2,\ell}E_4^{3n+6}\Delta^{s_\ell-n}
	-\sum_{n=-3}^{s_\ell-4}
		64c_{n+4,\ell}E_4^{3n+12}\Delta^{s_\ell-n}
	\nonumber\\
	&+(\chi_{12}(\ell)\ell(1+\ell)-c_{0,\ell})E_2\Delta^{s_\ell}	
	\pmod{128}
\nonumber\\
\equiv&
	\sum_{n=1}^{s_\ell}a_{n,\ell}E_4^{3n-1}E_6\Delta^{s_\ell-n}	
	+ a_{0,\ell}E_2\Delta^{s_\ell}
	\nonumber\\
	&-(32c_{1,\ell}+64c_{3,\ell})\Delta^{s_\ell+1}
	-64c_{2,\ell}\Delta^{s_\ell+2}
	-64c_{1,\ell}\Delta^{s_\ell+3}
	\pmod{128},
\end{align}
where the $a_{n,\ell}$ are integers.

Thus from the (\ref{XiMess}), (\ref{AByXiAndS}), and (\ref{AByIntegers}) we have
\begin{align}\label{SMess}
12&\ell S_\ell(z)\eta(z)\Delta(z)^{s_\ell}
\nonumber\\
\equiv&
	\sum_{n=1}^{s_\ell}(b_{n,\ell}-a_{n,\ell})E_4(z)^{3n-1}E_6(z)\Delta(z)^{s_\ell-n}
	-a_{0,\ell}E_2(z)\Delta(z)^{s_\ell}
	\nonumber\\	
	&+(32c_{1,\ell}+64c_{3,\ell})\Delta(z)^{s_\ell+1}
	+64c_{2,\ell}\Delta(z)^{s_\ell+2}
	+64c_{1,\ell}\Delta(z)^{s_\ell+3}
	\pmod{128}.
\end{align}

However, in $12\ell S_\ell(z)\eta(z)\Delta(z)^{s_\ell}$ the lowest possible power of $q$ is $s_\ell+1$, whereas
\begin{align}
	E_4^{3n-1}(z)E_6(z)\Delta(z)^{s_\ell-n} &= q^{s_\ell-n}+\dots,
	\\
	E_2\Delta^{s_\ell} &= q^{s_\ell}+\dots.
\end{align}
Thus the congruence in (\ref{SMess}) requires every term with $\Delta(z)^n$, for $0\le n\le s_\ell$, to vanish modulo $128$. Hence we have
\begin{align}
	&12\ell S_\ell(z)\eta(z)\Delta(z)^{s_\ell}
	\nonumber\\\label{congruenceForS}	
	&\equiv 
		(32c_{1,\ell}+64c_{3,\ell})\Delta^{s_\ell+1}(z)
		+64c_{2,\ell}\Delta^{s_\ell+2}(z)
		+64c_{1,\ell}\Delta^{s_\ell+3}(z)
	\pmod{128}.
\end{align}

We now must determine congruences satisfied by $c_{1,\ell},c_{2,\ell},$ and $c_{3,\ell}$. We recall the $c_{n,\ell}$ are given by
\begin{align}\label{EqForC}
	\sum_{n=0}^{s_\ell}c_{n,\ell}x^n &= \ell\chi_{12}(\ell)+A_{s_\ell}(x),
\end{align}
where
\begin{align}\label{DefinitionOfA}
	\sum_{n=0}A_m(x)q^m 
		&= \aqprod{q}{q}{\infty} \frac{E_4(z)^2E_6(z)}{\Delta(z)}
			\frac{1}{j(z)-x}.
\end{align}
In particular, $c_{1,\ell}$ is the coefficient of $xq^{s_\ell}$, $c_{2,\ell}$ is the coefficient of $x^2q^{s_\ell}$, and $c_{3,\ell}$ is the coefficient of $x^3q^{s_\ell}$ in (\ref{DefinitionOfA}).

\begin{proposition}\label{prop1}
If $\ell\ge 5$ is prime and $\ell\equiv 1,11,13,17,19,23\pmod{24}$, then $c_{1,\ell}\equiv 0\pmod{2}$.
\end{proposition}
\begin{proof}
The coefficient of $x$ in (\ref{DefinitionOfA}) is given by
\begin{align}\label{EqForX}
	\aqprod{q}{q}{\infty} \frac{E_4(z)^2E_6(z)}{\Delta(z)}\frac{1}{j(z)^2}
	&= \frac{\aqprod{q}{q}{\infty}E_6(z)\Delta(z)}{E_4(z)^4}	
	\nonumber\\
	&\equiv \aqprod{q}{q}{\infty}\Delta(z) \pmod{2}
	\nonumber\\
	&\equiv \sum_{k=-\infty}^{\infty}	q^{(3k^2-k)/2}
		\sum_{m=1}^{\infty} q^{(2m-1)^2} \pmod{2}.
\end{align}

The coefficient of $q^{s_\ell}$ in (\ref{EqForX}) occurs when $k$ and $m$ are integers, $m\ge 1$, and are a solution to
\begin{align}
	\frac{\ell^2-1}{24} &= \frac{3k^2-k}{2} + (2m-1)^2, 
	\\\label{prop1Eq1}
	\ell^2 &= (6k-1)^2 + 6(4m-2)^2. 
\end{align}

We  consider the binary quadratic form 
\begin{align}
	f(x,y) &= x^2 + 6y^2.
\end{align}
We see that $f(x,y)$ has discriminant $-24$. Computing the Legendre symbol we find for $\ell$ prime that 
\begin{align}
	\Jac{-24}{\ell}
	&=
	\left\{
   	\begin{array}{rl}
      	1  & \hspace{15pt}\ell\equiv 1,5,7,11 \pmod{24} \\
       	-1 & \hspace{15pt}\ell\equiv 13,17,19,23 \pmod{24}.
     	\end{array}
	\right.
\end{align}

First we consider $\ell$ prime such that $\ell\equiv 13,17,19,23$. Since $\Jac{-24}{\ell}=-1$, we know $f(x,y)$ does not represent $\ell^2$ with $x$ and $y$ relatively prime. Thus $f(x,y)$ only represents $\ell^2$ when $(x,y) = (\pm\ell,0)$, and so there are no solutions to (\ref{prop1Eq1}) with $m$ an integer. That is, for prime $\ell\equiv 13,17,19,23 \pmod{24}$, we know $c_{1,\ell}$ to be even. 

For the rest of the proof we only consider $\ell$ prime such that $\ell\equiv 1,11 \pmod{24}$. Here it is the case that $f(x,y)$ represents $\ell^2$ with $x$ and $y$ relatively prime. We will show that in these cases the substitution 
\begin{align}\label{prop1SubForY}
	4m-2 &= y 
\end{align}
also does not give an integer value for $m$.

The only other primitive form of discriminant $-24$ is given by
\begin{align}
	g(x,y)=2x^2+3y^2.
\end{align}
A computation shows that $g(x,y)$ is never congruent to $1 \pmod{24}$ and so $g$ does not represent $\ell^2$. Since $\Jac{-24}{\ell} = 1$ we have by Corollary 4.1 of Sun and Williams \cite{SunAndWilliams} that $\ell^2$ is represented exactly six times by $f(x,y)$. These six representations are from $(x,y)=(\pm \ell,0)$ and $(x,y)=(\pm x_0,\pm y_0)$ where $x_0,y_0$ are positive and relatively prime integers. We will show the binary quadratic form $h$, given by
\begin{align}
	h(x,y) &= x^2+96y^2,
\end{align}
also represents $\ell^2$ six times. With this we would know $y_0\equiv 0 \pmod{4}$ and so the substitution in 
(\ref{prop1SubForY}) does not give an integer value for $m$.

We let $(a,b,c)$ denote the binary quadratic form $ax^2+bxy+cy^2$ and $[a,b,c]$ denote the equivalence class of binary quadratic forms equivalent to $(a,b,c)$. 

The discriminant of $h(x,y)$ is $d=-384$. The primitive forms of discriminant $-384$, including $h$, are given by
\begin{align}\label{prop1Primitives}
	 &(1,0,96), (3,0,32), (4,4,25), (5,4,20), (5,-4,20), (7,6,15), (7,-6,15),
	\nonumber\\
	&\mbox{and }(11,10,11).
\end{align}
By Corollary 4.1 of \cite{SunAndWilliams} we know that altogether the forms in (\ref{prop1Primitives}) represent $\ell^2$ six times total. However, a computation shows that $(3,0,32)$, $(5,4,20)$, $(5,-4,20)$, $(7,6,15)$, $(7,-6,15)$, and $(11,10,11)$ are never congruent to $1 \pmod{24}$, hence these forms cannot represent $\ell^2$. The only remaining forms are then $(1,0,96)$ and $(4,4,25)$.
Since $\Jac{-24}{\ell}=1$ we do know at least one of the forms in (\ref{prop1Primitives}) represents $\ell$. 

In the case $\ell\equiv 1\pmod{24}$, we know $(1,0,96)$ or $(4,4,25)$ must represent $\ell$. However Dirichlet composition gives 
$[4,4,25]^2 = [1,0,96]$ and so regardless of which form represents $\ell$, Theorem 5.1 of \cite{SunAndWilliams} gives that $(1,0,96)$ represents $\ell^2$ six times.

In the case $\ell\equiv 11\pmod{24}$, we find that of the forms in (\ref{prop1Primitives}) only $(3,0,32)$ and $(11,10,11)$ are ever congruent to 
$11 \pmod{24}$, so it must be one of these two forms that represents $\ell$.
Here Dirichlet composition gives $[3,0,32]^2 = [1,0,96] = [11,10,11]^2$. In particular the forms $(3,0,32)$ and $(11,10,11)$ both represent elements of order two in the class group and $[4,4,25]$ is a power of neither. By Theorem 5.1 of \cite{SunAndWilliams} we have then that $(4,4,25)$ does not represent $\ell^2$. Thus all representations of $\ell^2$ must be from $x^2+96y^2$. 

Since $h(x,y)$ represents $\ell^2$ six times, the six representations of $\ell^2$ by $f(x,y)$ must have $y\equiv 0\pmod{4}$ and so the substitution in (\ref{prop1SubForY}) does not give an integer value for $m$. Hence there are no integer solutions to (\ref{prop1Eq1}) and so $c_{1,\ell}$ is even in the cases of $\ell\equiv 1,11\pmod{24}$ as well.

\end{proof}
That the form $x^2+96y^2$ represents $\ell^2$, with $x$ and $y$ non-zero, for $\ell\equiv 11\pmod{24}$ has also been proved and used in \cite{OnoAndPenniston1}, \cite{OnoAndPenniston2} by Ono and Penniston as part of obtaining a formula for the number, modulo $8$, of partitions of $n$ into distinct parts.

\begin{proposition}\label{prop2}
If $\ell\ge 5$ is prime and $\ell\equiv 1,5,11,17,19,23\pmod{24}$, then $c_{2,\ell}\equiv 0\pmod{2}$.
\end{proposition}
\begin{proof}
The proof is similar to the previous proposition and so we omit some of the details. The coefficient of $x^2$ in (\ref{DefinitionOfA}) is given by
{\allowdisplaybreaks\begin{align}\label{EqForX2}
	\aqprod{q}{q}{\infty} \frac{E_4(z)^2E_6(z)}{\Delta(z)}\frac{1}{j(z)^3}
	&= \frac{\aqprod{q}{q}{\infty}E_6(z)\Delta(z)^2}{E_4(z)^7}	
	\nonumber\\
	&\equiv \aqprod{q}{q}{\infty}\Delta(z)^2 \pmod{2}
	\nonumber\\
	&\equiv \aqprod{q}{q}{\infty}q^2\aqprod{q^{16}}{q^{16}}{\infty}^3 \pmod{2}
	\nonumber\\
	&\equiv \sum_{k=-\infty}^{\infty}	q^{(3k^2-k)/2}
		\sum_{m=0}^{\infty} q^{8m^2+8m+2} \pmod{2}.	
\end{align}}

The coefficient of $q^{s_\ell}$ in (\ref{EqForX2}) occurs when $k$ and $m$ are integers, $m\ge 0$, and are a solution to
\begin{align}
	\frac{\ell^2-1}{24} &= \frac{3k^2-k}{2} + 8m^2 +8m + 2,
	\\\label{prop2Eq1}	
	\ell^2 &= (6k-1)^2 + 3(8m+4)^2.
\end{align}

We consider the binary quadratic form 
\begin{align}
	f(x,y) &= x^2 + 3y^2.
\end{align} 
We see that $f(x,y)$ has discriminant $-12$. Computing the Legendre symbol we find for $\ell$ prime that
\begin{align}
	\Jac{-12}{\ell}
	&=
	\left\{
   	\begin{array}{rl}
      	1  & \hspace{15pt}\ell\equiv 1,7,13,19 \pmod{24} \\
       	-1 & \hspace{15pt}\ell\equiv 5,11,17,23 \pmod{24}.
     	\end{array}
	\right.
\end{align}
This proves the proposition for $\ell\equiv 5,11,17,23\pmod{24}$, however we still must handle the cases $\ell\equiv 1,19\pmod{24}$. 

We note that $f(x,y)$ is the only primitive form of discriminant $-12$. We also use the primitive form 
$h(x,y)=x^2+192y^2$ of discriminant $-768$. The primitive forms of discriminant $-768$ are given by
\begin{align}\label{prop2Primitives}
	&(1,0,192), (3,0,64), (4,4,49), (7,4,28), (7,-4,28), (13,8,16), (13,-8,16),
	\nonumber\\
	&\mbox{and } (12,12,19).
\end{align}

Again by using Corollary 4.1 and Theorem 5.1 of \cite{SunAndWilliams} we find there are six representations of $\ell^2$ by $f(x,y)$ and these correspond to six representations by $h(x,y)$. We then know the representations of $\ell^2$ by $f(x,y)$ have $y\equiv 0\pmod{8}$ and so there are no integer solutions to (\ref{prop2Eq1}). Hence $c_{2,\ell}$ is even in the cases of $\ell\equiv 1,19\pmod{24}$ as well.

\end{proof}

\begin{proposition}\label{prop3}
If $\ell\equiv 1,11,17,19\pmod{24}$ then $c_{1,\ell}+2c_{3,\ell}\equiv 0\pmod{4}$.
\end{proposition}
\begin{proof}The coefficient of $x$ in (\ref{DefinitionOfA}) is given by
\begin{align}
	\aqprod{q}{q}{\infty}\frac{E_4(z)^2E_6(z)}{\Delta(z)}\frac{1}{j(z)^2},
\end{align}
whereas the coefficient of $x^3$ in (\ref{DefinitionOfA}) is given by
\begin{align}
	\aqprod{q}{q}{\infty}\frac{E_4(z)^2E_6(z)}{\Delta(z)}\frac{1}{j(z)^4}.
\end{align}	
We note that
\begin{align}
	&\aqprod{q}{q}{\infty}\frac{E_4(z)^2E_6(z)}{\Delta(z)}\frac{1}{j(z)^2}
	+2\aqprod{q}{q}{\infty}\frac{E_4(z)^2E_6(z)}{\Delta(z)}\frac{1}{j(z)^4}	
	\nonumber\\&= \frac{\aqprod{q}{q}{\infty}E_6(z)\Delta(z)}{E_4(z)^4}	
		+2\frac{\aqprod{q}{q}{\infty}E_6(z)\Delta(z)^3}{E_4(z)^{10}}
	\nonumber\\
	&\equiv \aqprod{q}{q}{\infty}\Delta(z) 
		+ 2\aqprod{q}{q}{\infty}\Delta(z)^3
	\pmod{4}.
\end{align}
Thus $c_{1,\ell}+2c_{3,\ell}$ modulo $4$ is the coefficient of $q^{s_\ell}$ in
\begin{align}\label{C1AndC3Mess}
	\aqprod{q}{q}{\infty}\Delta(z) 
		+ 2\aqprod{q}{q}{\infty}\Delta(z)^3.
\end{align}
We suppose the series expansion of (\ref{C1AndC3Mess}) is give by
\begin{align}\label{prop3Eq1}
	\aqprod{q}{q}{\infty}\Delta(z) + 2\aqprod{q}{q}{\infty}\Delta(z)^3
	&= \sum_{n=1}^\infty a(n)q^n = q + \dots.
\end{align}

As in \cite{Ono1}, the Shimura lift of a half integral weight modular form is given as follows. Suppose that 
$g(z)=\sum_{n=1}^\infty b(n)q^n\in S_{\lambda+\frac{1}{2}}(\Gamma_0(4N),\chi)$ and $\lambda\ge 2$, then 
$\SLift{t}{\lambda}{g(z)} \in S_{2\lambda}(\Gamma_0(2N),\chi^2)$ where
\begin{align}
	\SLift{t}{\lambda}{g(z)} := \sum_{n=1}^\infty B_t(n)q^n,
	\\
	B_t(n) := \sum_{d\mid n}\psi_t(d)d^{\lambda-1}b\Parans{\frac{tn^2}{d^2}},
\end{align}
and $\psi_t$ is the Dirichlet character defined by $\psi_t(n) = \chi(n)\Jac{-1}{n}^\lambda\Jac{t}{n}$.

Consider 
\begin{align}
	g(z) &:= \eta(24z)\Delta(24z)E_4^6(24z)+ 2\eta(24z)\Delta(24z)^3,
\end{align}
this is a weight $36+\frac{1}{2}$ modular form on $\Gamma_0(576)$ with character
$\chi_{12}$. We suppose the $q$-series expansion of $g$ is
\begin{align}\label{prop3Eq2}
	g(z) &= \sum_{n=1}^\infty b(n) = q^{25} + \dots.
\end{align}

We recall that $s_\ell = \frac{\ell^2-1}{24}$ and so comparing the coefficients in (\ref{prop3Eq1}) and (\ref{prop3Eq2})
we have
\begin{align}
	b(\ell^2) &\equiv a(s_\ell) \pmod{4}.
\end{align}
Thus to prove the proposition, we can instead prove
\begin{align}\label{prop3Eq3}
	b(\ell^2) &\equiv 0 \pmod{4}
\end{align}
for primes $\ell\equiv 1,11,17,19 \pmod{24}$.

Now we consider the $t=1$ Shimura lift of $g(z)$, say
\begin{align}
	\SLift{1}{36}{g(z)} &= \sum_{n=1}^\infty B(n)q^n.
\end{align}
The $B(n)$ are given by
\begin{align}
	B(n) &= \sum_{d\mid n}\psi_1(d)d^{35}b\Parans{\frac{n^2}{d^2}},
\end{align}
where
\begin{align}
	\psi_1(n) &= \chi_{12}(n)\Jac{-1}{n}^{36}\Jac{1}{n}.
\end{align}
For $\ell$ prime we have
\begin{align}
	B(\ell) &= \psi_1(1)b(\ell^2) + \psi_1(\ell)\ell^{35}b(1)
	\nonumber\\
	&= b(\ell^2),
\end{align}
noting that $b(1) = 0$. We will show
\begin{align}
	B(n)\equiv 0 \pmod{4},
\end{align}
for all $n\equiv 1,11,17,19\pmod{24}$. This will then prove (\ref{prop3Eq3}).

We define
\begin{align}
	g_1(z) &:= \eta(24z)\Delta(24z),
	\\
	g_2(z) &:= 2\eta(24z)\Delta(24z)^3.
\end{align}
So $g(z)\equiv g_1(z)+g_2(z)\pmod{4}$ and since $d^{11}\equiv d^{35}\pmod{4}$ we have
\begin{align}\label{SLifts}
	\SLift{1}{36}{g(z)}\equiv \SLift{1}{12}{g_1(z)}+\SLift{1}{36}{g_2(z)}\pmod{4}.
\end{align}

For a given modular form $f(z)=\sum a(n)q^n$ we let $f\otimes\chi$ denote the twist of $f$ by a Dirichlet character $\chi$,
\begin{align}
 	f(z)\otimes\chi = \sum \chi(n)a(n)q^n. 
\end{align}
If $f\in S_{\lambda}(\Gamma_0(N),\psi)$,  and $\chi$ is a character modulo $m$, then as in \cite{Ono1} we know  $f\otimes\chi\in S_{\lambda}(\Gamma_0(Nm^2),\psi\chi^2)$. Furthermore, By Lemma 4.3.10 of \cite{Miyake}, if $\psi$ has conductor $m_\psi$, $\chi$ is a primitive character of conductor $m_\chi$, and $M$ is the least common multiple of 
$N$, $m_\chi^2$, and $m_\chi m_\psi$, then in fact
$f\otimes\chi\in S_{\lambda}(\Gamma_0(M),\psi\chi^2)$.
\\

We claim
\begin{align}\label{Weight24Lift}
	\SLift{1}{12}{g_1(z)}
	\equiv&
		\eta(2z)^{12}\eta(3z)^{24}\eta(6z)^{12}\otimes\chi_{12}
		\nonumber\\
		&-\eta(2z)^{24}\eta(3z)^{24}\otimes\chi_{12}
	\pmod{4}.
\end{align}
By Theorems 1.64 and  1.65 of \cite{Ono1} we know
$\eta(2z)^{12}\eta(3z)^{24}\eta(6z)^{12}$ and $\eta(2z)^{24}\eta(3z)^{24}$ are elements of $S_{24}(\Gamma_0(6))$.
Since $\chi_{12}$ is a primitive character with conductor $12$ we know the right hand side of $(\ref{Weight24Lift})$ is an element of $S_{24}(\Gamma_0(144))$. By Theorem 1 of Yang \cite{Yang} it turns out $\SLift{1}{12}{g_1(z)}$ is in fact also a cusp form on $\Gamma_0(6)$ twisted by $\chi_{12}$. Thus $\SLift{1}{12}{g_1(z)}$ is an element of $S_{24}(\Gamma_0(144))$ as well. To verify the congruence in (\ref{Weight24Lift}) we need only verify it holds out to the power of $q$ given by Sturm's theorem, which is $576$. This verification is done in Maple.

Next we claim
\begin{align}\label{Weight72Lift}
	\SLift{1}{36}{g_2(z)}
	\equiv&
	2\eta(z)^2\eta(2z)^2\eta(3z)^{10}\eta(6z)^{130}\otimes\chi_{12}\nonumber\\	
	&+2\eta(z)\eta(2z)^7\eta(3z)^5\eta(6z)^{131}\otimes\chi_{12}
	\nonumber\\
	&+2\eta(z)\eta(2z)^{-5}\eta(3z)^5\eta(6z)^{143}\otimes\chi_{12}
	\nonumber\\
	&+2\eta(z)^{-19}\eta(2z)^{-1}\eta(3z)^{25}\eta(6z)^{139}\otimes\chi_{12}
	\nonumber\\
	&+2\eta(2z)^{24}\eta(3z)^{120}\otimes\chi_{12}
	\nonumber\\
	&+2\eta(z)^{-31}\eta(2z)^{-1}\eta(3z)^{37}\eta(6z)^{139}\otimes\chi_{12}
	\nonumber\\
	&+2\eta(2z)^{132}\eta(3z)^{24}\eta(6z)^{-12}\otimes\chi_{12}
	\pmod{4}.
\end{align}
Again by Theorems 1.64 and  1.65 of \cite{Ono1} we know the right hand side of $(\ref{Weight72Lift})$ is an element of $S_{72}(\Gamma_0(144))$. Again by Theorem 1 of \cite{Yang} we know $\SLift{1}{36}{g_2(z)}$ to also be an element of $S_{72}(\Gamma_0(144))$. To verify the congruence in (\ref{Weight72Lift}) we need only verify it holds out to the power of $q$ given by Sturm's theorem, which is $1728$. This verification is done in Maple.

For a modular form $f(z) = \sum a(n)q^n$ we consider
\begin{align}
	P(f(z)) &= \sum_{n\equiv 1,11,17,19\pmod{24}}a(n)q^n.
\end{align}
We let $\psi_2$ denote the unique character modulo $2$, $\psi_3$ the unique primitive character modulo $3$, and $\psi_8$ the primitive character modulo $8$ given by
\begin{align}
	\psi_8(n) &=
	\left\{
   	\begin{array}{rl}
      	1  & \hspace{15pt} n \equiv 1,3\pmod{8} \\
       	-1 & \hspace{15pt} n \equiv 5,7\pmod{8}.
     	\end{array}
	\right.
\end{align} 
With this we find that
\begin{align}\label{prop3PByTwists}
	P(f(z)) &= \frac{1}{2}f(z)\otimes\psi_2\otimes\psi_3\otimes\psi_3 
			+\frac{1}{2}f(z)\otimes\psi_2\otimes\psi_3\otimes\psi_3\otimes\psi_8.
\end{align}

We recall that we're trying to show $P(\SLift{1}{36}{g(z)})\equiv 0\pmod{4}$. We define a modular form $h$ in $S_{72}(\Gamma_0(6))$ by 
\begin{align}\label{DefOfh}
	h(z) =& 
	2\eta(z)^2\eta(2z)^2\eta(3z)^{10}\eta(6z)^{130}\nonumber\\	
	&+2\eta(z)\eta(2z)^7\eta(3z)^5\eta(6z)^{131}\nonumber\\
	&+2\eta(z)\eta(2z)^{-5}\eta(3z)^5\eta(6z)^{143}\nonumber\\
	&+2\eta(z)^{-19}\eta(2z)^{-1}\eta(3z)^{25}\eta(6z)^{139}\nonumber\\
	&+2\eta(2z)^{24}\eta(3z)^{120}\nonumber\\
	&+2\eta(z)^{-31}\eta(2z)^{-1}\eta(3z)^{37}\eta(6z)^{139}\nonumber\\
	&+2\eta(2z)^{132}\eta(3z)^{24}\eta(6z)^{-12}\nonumber\\
	&+E_4^{12}(z)\eta(2z)^{12}\eta(3z)^{24}\eta(6z)^{12}\nonumber\\
	&-E_4^{12}(z)\eta(2z)^{24}\eta(3z)^{24}.
\end{align}
In particular by (\ref{SLifts}), (\ref{Weight24Lift}), and (\ref{Weight72Lift}) we have 
\begin{align}
	\SLift{1}{36}{g(z)} &\equiv h(z)\otimes\chi_{12} \pmod{4}.
\end{align}
We see that if $P(h(z))\equiv 0 \pmod{4}$ then $P(h(z)\otimes\chi_{12})\equiv 0 \pmod{4}$, and so in turn we would have $P(\SLift{1}{36}{g(z)})\equiv 0 \pmod{4}$. Thus to prove the proposition it is sufficient to prove 
\begin{align}\label{prop3LastEq}
	P(h(z))\equiv 0 \pmod{4}.
\end{align}

By (\ref{prop3PByTwists}) and (\ref{DefOfh}) we see that $P(h(z))$ is an element of $S_{72}(\Gamma_0(576))$. To verify the congruence in (\ref{prop3LastEq}) we need only verify it holds up to the power of $q$ specified by Sturm's theorem, which is $6912$. This verification is done in Maple and completes the proof of the proposition.
\end{proof}

By Propositions \ref{prop1}, \ref{prop2}, \ref{prop3} and (\ref{congruenceForS}) we have
\begin{align}
	12\ell S_\ell(z)\eta(z)\Delta(z)^{s_\ell} &\equiv 0\pmod{64} &\mbox{ for }\ell\equiv 13,23\pmod{24},\\
	12\ell S_\ell(z)\eta(z)\Delta(z)^{s_\ell} &\equiv 0\pmod{128} &\mbox{ for }\ell\equiv 1,11,17,19\pmod{24}.
\end{align}
Therefore $S_\ell(z) \equiv 0\pmod{16}$ for $\ell\equiv 13,23\pmod{24}$ and 
$S_\ell(z) \equiv 0\pmod{32}$ for $\ell\equiv 1,11,17,19\pmod{24}$. This completes the proof of Theorem \ref{MainTheorem}.

\bibliographystyle{abbrv}
\bibliography{arxivTypeUpRef}

\begin{thebibliography}{10}

\bibitem{ABL}
S.~Ahlgren, K.~Bringmann, and J.~Lovejoy.
\newblock {$\ell$}-adic properties of smallest parts functions.
\newblock {\em Adv. Math.}, 228(1):629--645, 2011.

\bibitem{AK}
S.~Ahlgren and B.~Kim.
\newblock Mock modular grids and {H}ecke relations for mock modular forms, to
  appear in ${F}orum$ ${M}ath.$; doi: 10.1515/forum-2012-0011.

\bibitem{Andersen}
N.~Andersen.
\newblock Hecke-type congruences for two smallest parts functions.
\newblock {\em Int. J. Number Theory}, 09(03):713--728, 2013.

\bibitem{Andrews}
G.~E. Andrews.
\newblock The number of smallest parts in the partitions of {$n$}.
\newblock {\em J. Reine Angew. Math.}, 624:133--142, 2008.

\bibitem{Atkin}
A.~O.~L. Atkin.
\newblock Multiplicative congruence properties and density problems for
  {$p(n)$}.
\newblock {\em Proc. London Math. Soc. (3)}, 18:563--576, 1968.

\bibitem{Bringmann}
K.~Bringmann.
\newblock On the explicit construction of higher deformations of partition
  statistics.
\newblock {\em Duke Math. J.}, 144(2):195--233, 2008.

\bibitem{Garvan3}
F.~G. Garvan.
\newblock Congruences for {A}ndrews' spt-function modulo powers of {$5$}, {$7$}
  and {$13$}.
\newblock {\em Trans. Amer. Math. Soc.}, 364(9):4847--4873, 2012.

\bibitem{Garvan1}
F.~G. Garvan.
\newblock Congruences for {A}ndrews' spt-function modulo 32760 and extension of
  {A}tkin's {H}ecke-type partition congruences.
\newblock In {\em Number Theory and Related Fields}, volume~43 of {\em Springer
  Proceedings in Mathematics and Statistics}. Springer, 2013.

\bibitem{Miyake}
T.~Miyake.
\newblock {\em Modular forms}.
\newblock Springer Monographs in Mathematics. Springer-Verlag, Berlin, english
  edition, 2006.
\newblock Translated from the 1976 Japanese original by Yoshitaka Maeda.

\bibitem{Ono1}
K.~Ono.
\newblock {\em The web of modularity: arithmetic of the coefficients of modular
  forms and {$q$}-series}, volume 102 of {\em CBMS Regional Conference Series
  in Mathematics}.
\newblock Published for the Conference Board of the Mathematical Sciences,
  Washington, DC, 2004.

\bibitem{Ono3}
K.~Ono.
\newblock Congruences for the andrews spt function.
\newblock {\em Proc. Natl. Acad. Sci. USA}, 108(2):473--476, 2011.

\bibitem{Ono2}
K.~Ono.
\newblock The partition function and {H}ecke operators.
\newblock {\em Adv. Math.}, 228(1):527--534, 2011.

\bibitem{OnoAndPenniston1}
K.~Ono and D.~Penniston.
\newblock The 2-adic behavior of the number of partitions into distinct parts.
\newblock {\em J. Combin. Theory Ser. A}, 92(2):138--157, 2000.

\bibitem{OnoAndPenniston2}
K.~Ono and D.~Penniston.
\newblock Corrigendum: ``{T}he 2-adic behavior of the number of partitions into
  distinct parts''.
\newblock {\em J. Combin. Theory Ser. A}, 95(1):196, 2001.

\bibitem{SunAndWilliams}
Z.-H. Sun and K.~S. Williams.
\newblock On the number of representations of {$n$} by {$ax^2+bxy+cy^2$}.
\newblock {\em Acta Arith.}, 122(2):101--171, 2006.

\bibitem{Yang}
Y.~Yang.
\newblock Modular forms of half-integral weights on ${SL}(2,\mathbb{Z})$.
\newblock arXiv:1110.1810, October 2011.

\end{thebibliography}

\end{document}